\documentclass[a4paper,12pt]{article}
\usepackage{amsfonts,amsmath,amssymb,amsthm}
\usepackage[hidelinks]{hyperref}

\newtheorem{thm}{Theorem}[section]
\newtheorem{cor}[thm]{Corollary}
\newtheorem{prop}[thm]{Proposition}
\newtheorem{lem}[thm]{Lemma}

\theoremstyle{definition}
\newtheorem{defn}[thm]{Definition}
\newtheorem{rem}[thm]{Remark}
\newtheorem{exas}[thm]{Examples}

\newcommand{\smsum}{\textstyle\sum\limits}

\begin{document}

\title{Simpleness of Leavitt Path Algebras with Coefficients in a
  Commutative Semiring}

\author{Y.~Katsov$^{1}$, T.\,G.~Nam$^{2}$, J.~Zumbr\"{a}gel$^{3}$\\
{\footnotesize katsov@hanover.edu; tgnam@math.ac.vn; jens.zumbragel@ucd.ie}\\
$^{1}$\footnotesize{Department of Mathematics}\\[-1.5mm]
\footnotesize{Hanover College, Hanover, IN 47243--0890, USA}\\
$^{2}${\footnotesize Institute of Mathematics, VAST}\\[-1.5mm]
{\footnotesize 18 Hoang Quoc Viet, Cau Giay, Hanoi, Vietnam}\\
$^{3}${\footnotesize Institute of Algebra}\\[-1.5mm]
{\footnotesize Dresden University of Technology, Germany}}

\date{}

\maketitle

\begin{abstract}
  In this paper, we study ideal- and congruence-simpleness for the
  Leavitt path algebras of directed graphs with coefficients in a
  commutative semiring~$S$, as well as establish some fundamental
  properties of those algebras. We provide a complete characterization
  of ideal-simple Leavitt path algebras with coefficients in a
  semifield~$S$ that extends the well-known characterizations when the
  ground semiring~$S$ is a field. Also, extending the well-known
  characterizations when~$S$ is a field or commutative ring, we
  present a complete characterization of congruence-simple Leavitt
  path algebras over row-finite graphs with coefficients in a
  commutative semiring~$S$.\medskip


  \footnote{%
    The second author is supported by the Vietnam National Foundation
    for Science and Technology Development (NAFOSTED).  The third
    author is supported by the Irish Research Council under Research
    Grant ELEVATEPD/2013/82.}
\end{abstract}


\section{Introduction}

In some way, ``prehistorical'' beginning of Leavitt path algebras
started with Leavitt algebras (\cite{leav:tmtoar} and \cite{leav:tmtohi}),
Bergman algebras (\cite{bergman:casurc}), and graph C$^{\ast}$-algebras 
(\cite{cunt:acgby}), considering rings with the \textit{Invariant
  Basis Number} property, universal ring constructions, and the
structure of a separable simple infinite C$^{\ast}$-algebra,
respectively.  As to the algebraic structures known as \textit{Leavitt
  path algebras} themselves, they were initiated and developed
independently, and using different approaches, in the foundational
papers on the subject \cite{ap:tlpaoag05} and \cite{amp:nktfga}. Then,
during the last decade, these algebras have continuously been of
significant interest to mathematicians from different areas of
mathematics such as ring and group theorists, analysts working in
C$^{\ast}$-algebras, and symbolic dynamicists, for example. For a
detailed history and overview of the Leavitt path algebras we refer
our potential readers to a recent quite remarkable and motivating
survey on the subject~\cite{a:lpatfd}.

In our time, we may clearly observe a steadily growing interest in
developing algebraic and homological theories of semirings and
semimodules, as well as in their numerous connections with, and
applications in, different branches of mathematics, computer science,
cryptography, quantum physics, and many other areas of science (see,
\textit{e.g.}, \cite{gla:agttlosataimais}). As is well known,
structure theories for varieties of algebras constitute an important
``classical'' area of the sustained interest in algebraic research. In
those theories, so-called simple algebras, \textit{i.e.}, algebras
possessing only two trivial congruences -- the identity and universal
ones -- play a very important role of ``building blocks.'' In
addition, simple semirings, constituting another booming area of
semiring research, have quite interesting and promising applications
in various fields, in particular in cryptography (see, \textit{e.g.},
\cite{mmr:pkcbosa}). However, in contrast to the varieties of groups
and rings, research on simple semirings has been started only
recently, and therefore not much on the subject is known. Also,
investigating semirings and their representations, one should
undoubtedly use methods and techniques of both ring and lattice theory
as well as diverse techniques and methods of categorical and universal
algebra, and work in a ``nonabelian environment.''  Perhaps all these
circumstances explain why research on simple semirings is still behind
of that for rings and groups (for some recent activity and results on
this subject one may consult \cite{mf:ccs}, \cite{bshhurtjankepka:scs},
\cite{monico:ofcss}, \cite{bashkepka:css}, \cite{zumbr:cofcsswz},
\cite{jezkepkamaroti:tesoas}, \cite{knt:mosssparp}, \cite{kz:fsais},
\cite{knz:ososacs}).

Motivated by \cite{a:lpatfd}, in this paper we initiate a study of
Leavitt path algebras (it deserves to be mentioned that, in some way,
a generalization of an idea of Leavitt algebras from \cite{leav:tmtoar}
to a semiring setting was earlier considered in \cite{hebwei:otrosos})
in a nonadditive/nonabelian semiring setting --- working with
semirings and semimodules, we live in a ``world without subtraction'' and,
therefore, have no privilege of the classical well developed
techniques of additive/abelian categories of modules over rings.  More
precisely, we consider the concepts of Leavitt path algebras with
coefficients in a commutative semiring~$S$, and of ideal- and
congruence-simpleness for those algebras; note that in our semiring
setting, in contrast to the ``additive'' ring case, these two
notions of simpleness are not the same (see, \textit{e.g.}, \cite[%
Examples~3.8]{knz:ososacs}) and should be differed. In light of this,
presenting some new, important and interesting in our view,
considerations, results and techniques regarding characterizations of
ideal- and congruence-simple Leavitt path algebras over a commutative
ground semiring~$S$, extending the ``classical'' ring characterizations 
(see, \cite[Theorem~3.11]{ap:tlpaoag05}, \cite[Theorem~3.1]{ap:tlpaoag08},
\cite[Theorem~6.18]{tomf:utaisflpa} and \cite[Theorem~3.11]{g:lpaadl}),
as well as motivating an interest to this direction of research, is a
main goal of our paper.

For the reader's convenience, all subsequently necessary basic
concepts and facts on semirings and Leavitt path algebras with
coefficients in a commutative semiring are collected in
Section~\ref{sec:basic}.

In Section~\ref{sec:ideal}, together with establishing some
important properties of the Leavitt path algebras with coefficients in
a commutative semiring~$S$, we provide a complete characterization of
ideal-simple Leavitt path algebras with coefficients in a semifield~$S$
(Theorem~3.4), constituting one of the central results of the paper
and extending the well-known characterizations (see,
\cite[Theorem~3.11]{ap:tlpaoag05}, \cite[Theorem~3.1]{ap:tlpaoag08},
\cite[Theorem~6.18]{tomf:utaisflpa} and \cite[Theorem~3.11]{g:lpaadl})
when the ground semiring $S$ is a field.

In Section~\ref{sec:congr}, together with establishing some
fundamental facts about the Leavitt path algebras with coefficients in
the Boolean semifield~$\mathbf{B}$ and combining them with
Theorem~3.4, we present a complete characterization of
congruence-simple Leavitt path algebras over row-finite graphs with
coefficients in a commutative semiring~$S$ (Theorem~4.5), constituting
another main result of the paper and extending the well-known
characterizations from [\textbf{op.\ cit.}]. It should be emphasized
that, in contrast to the ``classical'' case of the ground
structure~$S$ to be a commutative ring, in order to establish these
results in our semiring setting, one needs to exploit some innovative
approach and techniques of universal algebra based on dealing with
congruences rather then with ideals. Also, resolving
\cite[Problem~2]{knt:mosssparp} in the class of Leavitt path algebras
with coefficients in a commutative semiring~$S$, we show
(Corollary~4.2) that for algebras of this class the
congruence-simpleness implies their ideal-simpleness.

Finally, for notions and facts from semiring theory, we refer to
\cite{golan:sata}.


\section{Basic concepts}\label{sec:basic}

\subsection{Preliminaries on semirings}

Recall \cite{golan:sata} that a \emph{hemiring\/} is an algebra $(S,+,\cdot
,0)$ such that the following conditions are satisfied:

(1) $(S,+,0)$ is a commutative monoid with identity element $0$;

(2) $(S,\cdot)$ is a semigroup;

(3) Multiplication distributes over addition from either side;

(4) $0s=0=s0$ for all $s\in S$.

A hemiring~$S$ is \emph{commutative} if $(S, \cdot)$ is a commutative
semigroup; and a hemiring $S$ is \emph{additively idempotent} if
$a+a=a$ for all $a\in S$.  Moreover, a hemiring $S$ is a \emph{semiring} 
if its multiplicative semigroup $(S, \cdot)$ actually is a monoid
$(S, \cdot, 1)$ with identity element~$1$.  A commutative semiring~$S$
is a \emph{semifield} if $(S \setminus \{0\}, \cdot, 1)$ is a group.  Two
well-known examples of semifields are the additively idempotent two
element semiring $\mathbf{B} = \{0,1\}$, the so-called \emph{Boolean
  semifield}, and the tropical semifield $(\mathbb{R}\cup \{-\infty
\},\vee ,+,-\infty ,0\})$.

As usual, given two hemirings $S$ and $S'$, a map $\varphi: S 
\longrightarrow S'$ is a \emph{homomorphism} iff $\varphi(x+y) = 
\varphi(x) + \varphi(y)$ for all $x,y\in S$; and a submonoid~$I$ of 
$(S,+,0)$ is an \emph{ideal} of a hemiring~$S$ iff $sa$ and $as\in I$ for
all $a \in I$ and $s \in S$; an equivalence relation $\rho$ on a hemiring~$S$
is a \emph{congruence} iff $(s+a,s+b)\in \rho v$, $(sa,sb)\in \rho $ and 
$(as, bs) \in \rho$ for all pairs $(a,b) \in \rho $ and $s \in S$.  On every
hemiring $S$ there are always the two trivial congruences --- the \emph{%
diagonal congruence}, $\vartriangle_{S} := \{(s,s) \mid s\in S\}$, and the
\emph{universal congruence}, $S^{2} := \{(a,b) \mid a,b\in S\}$.  Following 
\cite{bshhurtjankepka:scs}, a hemiring~$S$ is \textit{congruence-simple} if 
$\vartriangle_{S}$ and $S^{2}$ are the only congruences on~$S$;
and~$S$ is \textit{ideal-simple} if $0$ and~$S$ are the only ideals
of~$S$.  It is clear that a hemiring~$S$ is congruence-simple iff
every nonzero hemiring homomorphism $\varphi: S\longrightarrow S'$ is
injective.  Obviously, the concepts congruence- and ideal-simpleness
are the same for rings and, therefore, we have just simple rings, but
they are different ones for semirings in general (see, \textit{e.g.},
\cite[Examples 3.8]%
{knz:ososacs}).

An $S$-\emph{semimodule} over a given commutative semiring~$S$ is a
commutative monoid $(M,+,0_{M})$ together with a scalar multiplication 
$(s,m) \mapsto sm$ from $S\times M$ to~$M$ which satisfies the identities 
$(ss')m = s(s'm)$, $s(m+m') = sm+sm'$, $(s+s')m = sm+s'm$, $1m=m$, 
$s0_{M} = 0_{M} = 0m$ for all $s,s'\in S$ and $m,m'\in M$. \emph{Homomorphisms}
between semimodules and \emph{free} semimodules are defined in the standard
manner.

By an $S$-algebra $A$ over a given commutative semiring $S$ we mean an 
$S$-semimodule $A$ with an associative bilinear $S$-semimodule multiplication
``\,$\cdot$\,'' on $A$.  An $S$-algebra~$A$ is \emph{unital} if $(A,\cdot)$
is actually a monoid with a neutral element $1_{A}\in A$, \textit{i.e.}, 
$a1_{A}=a=1_{A}a$ for all $a\in A$.  For example, every hemiring is an 
$\mathbb{N}$-algebra, where $\mathbb{N}$ is the semiring of the natural 
numbers with added $0$; and, of course, every additively idempotent hemiring 
is a $\mathbf{B}$-algebra.

Let $S$ be a commutative semiring and $\{x_{i} \mid i\in I\}$ a set of
independent, noncommuting indeterminates. Then $S\langle x_{i} \mid i\in
I\rangle $ will denote the free $S$-algebra generated by the indeterminates 
$\{x_{i} \mid i\in I\}$, whose elements are polynomials in the noncommuting
variables $\{x_{i} \mid i\in I\}$ with coefficients from $S$ that commute
with each variable $x_{i},i\in I$.

Finally, let~$S$ be a commutative semiring and $(G,\cdot ,1)$ a group. Then
we can form the \emph{group semiring} $S[G]$, whose elements are formal sums
$\sum_{g\in G}a_{g}g$ with the \emph{coefficients} $a_{g}\in $ $S$ and the
finite \emph{support}, \textit{i.e.}, almost all $a_{g}=0$. As usual, the
operations of addition and multiplication on $S[G]$ are defined as follows
\begin{gather*}
  \smsum_{g\in G}a_{g}g + \smsum_{g\in G}b_{g}g = \smsum_{g\in G}
  (a_{g}+b_{g})g , \\
  (\smsum_{g\in G}a_{g}g) (\smsum_{h\in G}b_{h}h) = \smsum_{t\in G}c_{t}t ,
\end{gather*}%
where $c_{t}=\sum a_{g}b_{h}$, with summation over all $(g,h)\in G\times G$
such that $gh=t$. Clearly, the elements of $S:=$ $S\cdot 1$ commute with the
elements of $G:=1\cdot G$ under the multiplication in $S[G]$. In particular,
one may easily see that $S[\mathbb{Z}]\cong S[x,x^{-1}]$, where $S[x,x^{-1}]$
is the algebra of the \textit{Laurent polynomials} over~$S$.

\subsection{Basics on Leavitt path algebras with coefficients in 
  a commutative semiring}

In this subsection, we introduce Leavitt path algebras having coefficients
in an arbitrary commutative semiring~$S$.  The construction of such algebras
is, certainly, a straightforward generalization of the constructions of the
Leavitt path algebras with the semiring~$S$ to be a field and a commutative
ring with unit originated in~\cite{ap:tlpaoag05} and~\cite{tomf:lpawciacr},
respectively.  All these constructions are crucially based on some general
notions of graph theory that for the reader's convenience we reproduce here.

A (directed) graph $\Gamma =(V,E,s,r)$ consists of two disjoint sets~$V$ 
and~$E$ -- \textit{vertices} and \textit{edges}, respectively -- and two 
maps $s,r: E\longrightarrow V$.  If $e\in E$, then $s(e)$ and $r(e)$ are 
called the \textit{source} and \textit{range} of $e$, respectively.  The 
graph $\Gamma$ is \emph{row-finite} if $|s^{-1}(v)| < \infty$ for every 
$v \in V$.  A vertex~$v$ for which $s^{-1}(v)$ is empty is called a
\emph{sink}; and a vertex~$v$ is \emph{regular} iff $0 < |s^{-1}(v)| < \infty$.
A \emph{path} $p = e_{1} \dots e_{n}$ in a graph $\Gamma$ is a sequence of 
edges $e_{1}, \dots, e_{n}$ such that $r(e_{i}) = s(e_{i+1})$ for
$i = 1, \dots, n-1$.  In this case, we say that the path~$p$ starts at
the vertex $s(p) := s(e_{1})$ and ends at the vertex $r(p) := r(e_{n})$,
and has \emph{length} $|p| := n$.  We consider the vertices in~$V$ to be
paths of length~$0$.  If $s(p) = r(p)$, then~$p$ is a \emph{closed path
  based at} $v = s(p) = r(p)$.  Denote by $\text{CP}(v)$ the set of all
such paths.  A closed path based at~$v$, $p = e_{1} \dots e_{n}$, is a
\emph{closed simple path based at}~$v$ if $s(e_i) \neq v$ for every $i > 1$.
Denote by $\text{CSP}(v)$ the set of all such paths.  If $p = e_{1}
\dots e_{n}$ is a closed path and all vertices $s(e_{1}), \dots, s(e_{n})$
are distinct, then the subgraph $(s(e_{1}), \dots, s(e_{n}); e_{1},
\dots, e_{n})$ of the graph $\Gamma$ is called a \emph{cycle}.  An
edge~$f$ is an \emph{exit} for a path $p = e_{1} \dots e_{n}$ if
$s(f) = s(e_{i})$ but $f \ne e_{i}$ for some $1 \le i \le n$.

\begin{defn}[{\textit{cf.}~\cite[Definition~1.3]{ap:tlpaoag05} and 
    \cite[Definition~2.4]{tomf:lpawciacr}}] 
  Let $\Gamma = (V,E,s,r)$ be a graph and $S$ a commutative semiring.
  The \emph{Leavitt path algebra} $L_{S}(\Gamma)$ of the graph~$\Gamma$
  \emph{with coefficients in}~$S$ is the $S$-algebra presented by the
  set of generators $V\cup E\cup E^{\ast}$ -- where $E\rightarrow E^{\ast}$, 
  $e\mapsto e^{\ast}$, is a bijection with $V$, $E$, $E^{\ast}$ pairwise
  disjoint -- satisfying the following relations:
  
  (1) $v v' = \delta_{v,v'} v$ for all $v, v' \in V$;
  
  (2) $s(e) e = e = e r(e)$, $r(e) e^{\ast} = e^{\ast} = e^{\ast} s(e)$ for
  all $e\in E$;
  
  (3) $e^{\ast} f = \delta_{e,f} r(e)$ for all $e,f \in E$;
  
  (4) $v = \sum_{e\in s^{-1}(v)} e e^{\ast}$ whenever $v\in V$ is a regular
  vertex.
\end{defn}

It is easy to see that the mappings given by $v \mapsto v$, for $v \in V$,
and $e \longmapsto e^{\ast}$, $e^{\ast} \longmapsto e$ for $e\in E$, produce 
an involution on the algebra $L_{S}(\Gamma)$, and for any path $p = e_{1} 
\dots e_{n} $ there exists $p^{\ast} := e_{n}^{\ast} \dots e_{1}^{\ast}$.

Observe that the Leavitt path algebra $L_{S}(\Gamma)$ can also be
defined as the quotient of the free $S$-algebra $S\langle v,e,e^{\ast}
\mid v\in V, e\in E, {e^{\ast}\in E^{\ast}} \rangle$ by the congruence
$\sim$ generated by the following ordered pairs:

(1) $(v v', \delta_{v,v'} v)$ for all $v, v' \in V$,

(2) $(s(e) e, e), (e, e r(e))$ and $(r(e) e^{\ast}, e^{\ast}), (e^{\ast},
e^{\ast} s(e))$ for all $e \in E$,

(3) $(e^* f, \delta_{e,f} r(e))$ for all $e, f \in E$,

(4) $(v, \sum_{e\in s^{-1}(v)} e e^{\ast})$ for all regular vertices $v\in V$.

\begin{rem}
  As will be shown in Proposition~2.4, for any graph $\Gamma =
  (V,E,s,r)$, all generators $\{ v, e, e^{\ast} \mid v\in V, e\in E,
  e^{\ast}\in E^{\ast}\}$ of $L_{S}(\Gamma)$ are nonzero.
  Furthermore, from the observation above, it readily follows that
  $L_{S}(\Gamma)$ is, in fact, the ``largest'' algebra generated by
  the elements $\{v, e, e^{\ast} \mid v\in V, e\in E, e^{\ast}\in E^{\ast}\}$
  satisfying the relations (1) -- (4) of Definition~2.1, in other words, 
  $L_{S}(\Gamma)$ has the following \textit{universal} property: 
  If~$A$ is an $S$-algebra generated by a family of elements
  $\{a_{v}, b_{e}, c_{e^{\ast}} \mid c\in V, e\in E, {e^{\ast } \in E^{\ast}}\}$ 
  satisfying the analogous to (1) -- (4) relations in Definition~2.1,
  then there always exists an $S$-algebra homomorphism 
  $\varphi: L_{S}(\Gamma) \rightarrow A$ given by ${\varphi(v) = a_{v}}$,
  ${\varphi(e) = b_{e}}$ and ${\varphi(e^{\ast}) = c_{e^{\ast}}}$.
\end{rem}

The following examples illustrate that some well-known (classical) algebras
actually can be viewed as the Leavitt path algebras as well.

\begin{exas}[{\textit{cf}.~\cite[Examples~1.4]{ap:tlpaoag05}}]
  Let $S$ be a commutative semiring.

  (i) Let $\Gamma = (V,E,s,r)$ be a graph with $V = \{v_{1}, \dots,
  v_{n}\}$ and $E = \{e_{1}, \dots, e_{n-1}\}$, where $s(e_{i}) =
  v_{i}$, $r(e_{i}) = v_{i+1}$ for all $i = 1, \dots, n-1$.  Then it is
  easy to check that the map $\varphi: L_{S}(\Gamma) \longrightarrow
  M_{n}(S)$, given by $\varphi(v_{i}) = E_{i,i}$, $\varphi(e_{i}) = E_{i,i+1}$
  and $\varphi(e_{i}^{\ast}) = E_{i+1,i}$, where $\{E_{i,j} \mid 1 \leq i,j
  \leq n\}$ are the standard elementary matrices in the $n \times n$ matrix
  semiring $M_{n}(S)$, is an $S$-algebra isomorphism.

  (ii) Let $\Gamma = (V,E,s,r)$ be a graph given by $V = \{ v \}$ and
  $E = \{ e \}$.  Then it is obvious that the Leavitt path algebra
  $L_{S}(\Gamma)$ is isomorphic to the Laurent polynomial algebra
  $S[x,x^{-1}]$ with $x:=e$ and $x^{-1} := e^{\ast}$.

  (iii) In \cite{leav:tmtohi}, investigating rings with the Invariant
  Basis Number property there were introduced what we now call the
  Leavitt algebras of the form $L_{K}(1,n)$, where $K$ is a field and
  $n \geq 2$ is a natural number, of type $(1,n)$.  Then, in
  \cite{hebwei:otrosos}, the authors, generalizing the Leavitt algebra
  construction in a semiring setting, constructed an $S$-algebra 
  $L_{S}(1,n)$ over a commutative semiring~$S$ which was defined by
  the generators $\{x_{i}, y_{i} \mid 1 \leq i \leq n\}$ and relations
  $x_{i} y_{j} = \delta_{ij}$ for all $1 \leq i,j \leq n$, and 
  $\sum_{i=1}^{n} y_{i} x_{i} = 1$.  Considering the graph $\Gamma =
  (V,E,s,r)$ given by $V = \{v\}$ and $E = \{e_{1}, \dots, e_{n}\}$, one
  may easily verify that the Leavitt path algebra $L_{S}(\Gamma)$ is,
  in fact, isomorphic to $L_{S}(1,n)$ by letting $y_{i} := e_{i}$ and 
  $x_{i} := e_{i}^{\ast}$ for all $1 \leq i \leq n$.
\end{exas}

The following proposition is an analog of \cite[Proposition~3.4]%
{tomf:lpawciacr} for a non-abelian semiring setting and presents some
fundamental properties of the Leavitt path algebras.

\begin{prop}[{\textit{cf}.~\cite[Proposition~3.4]{tomf:lpawciacr}}]
 Let $\Gamma = (V,E,s,r)$ be a graph and~$S$ a commutative semiring.  Then,
 the Leavitt path algebra $L_{S}(\Gamma)$ has the following properties:

   (1) All elements of the set $\{ v, e, e^{\ast} \mid v\in V,e\in E,
   e^{\ast} \in E^{\ast}\}$ are nonzero;
   
   (2) If $a, b$ are distinct elements in~$S$, then $av\neq bv$ for all 
   $v\in V$;

   (3) Every monomial in $L_{S}(\Gamma)$ is of the form $\lambda p q^{\ast}$,
   where $\lambda \in S$ and $p, q$ are paths in $\Gamma$ such that
   $r(p) = r(q)$.
\end{prop}

\begin{proof} 
  The proof given for the case of rings in \cite[Proposition~3.4]%
  {tomf:lpawciacr}, which, in turn, uses a similar construction as for
  the case of fields from \cite[Lemma~1.5]{g:lpaadl}, is based on
  Remark~2.2 --- there should be constructed an $S$-algebra $A$ as in
  Remark~2.2 having all generators $\{a_{v}, b_{e}, c_{e^{\ast}} \mid
  v\in V, e\in E, e^{\ast}\in E^{\ast}\}$ to be nonzero.  It almost
  does not depend on the ``abelianness'' of the ring case and,
  therefore, it works in our semiring setting as well.  Just for the
  reader's convenience, we have decided to sketch it here.

  Thus, let~$I$ be an infinite set of the cardinality at least $|V \cup E|$,
  and let $Z:=S^{(I)}$ a free $S$-semimodule with the basis $I$, \textit{i.e.},
  $Z$ is a direct sum of~$|I|$ copies of~$S$. For each $e\in E$, let 
  $A_{e}:=Z$ and, for each $v\in V$, let
  \[ A_{v} := \begin{cases}
    \bigoplus_{s(e)=v} A_{e} & \text{if } |s^{-1}(v)| \ne \varnothing , \\
    Z & \text{if } v \text{ is a sink.}%
  \end{cases} \]
  
  Note that all~$A_{e}$ and~$A_{v}$ are all mutually isomorphic, since each 
  of them is the direct sum of~$|I|$ many copies of~$S$.  Let $A :=
  \bigoplus_{v\in V} A_{v}$.  For each $v \in V$ define $T_{v}: A_{v} 
  \longrightarrow A_{v}$ to be the identity map and extend it to a 
  homomorphism $T_{v}: A \longrightarrow A$ by defining~$T_{v}$ to be zero 
  on $A \ominus A_{v}$.  Also, for each $e \in V$ choose an isomorphism 
  $T_{e}: A_{r(e)} \longrightarrow A_{e} \subseteq A_{s(e)}$ and extend it to a 
  homomorphism $T_{e}: A \longrightarrow A$ by mapping to zero on 
  $A \ominus A_{r(e)}$.  Finally, we define $T_{e^{\ast}}: A \longrightarrow A$ 
  by taking the isomorphism $T_{e}^{-1}: A_{e} \subseteq A_{s(e)} \longrightarrow
  A_{r(e)}$ and extending it to a homomorphism $T_{e^{\ast}}: A \longrightarrow A$
  by letting $T_{e^{\ast}}$ to be zero on $A \ominus A_{e}$.

  Now consider the subalgebra of $\text{Hom}_{S}(A,A)$ generated by $\{T_{v},
  T_{e}, T_{e^{\ast}} \mid v\in V, e\in E, e^{\ast}\in E^{\ast}\}$.  It is 
  straightforward to check (\textit{cf.}~\cite[Lemma~1.5]{g:lpaadl}) that 
  $\{T_{v},T_{e},T_{e^{\ast}} \mid {v\in V}, {e\in E}, {e^{\ast}\in E^{\ast}}\}$ 
  is a collection of nonzero elements satisfying the relations described
  in Definition~2.1.  By the universal property of $L_{S}(\Gamma)$, we
  get that the elements of the set $\{v,e,e^{\ast} \mid {v\in V}, {e\in E},
  {e^{\ast}\in E^{\ast}}\}$ are nonzero and (1) is established.
  
  Next we note that for each $v\in V$ we have $A_{v}=S\oplus M$ for some $S$%
  -semimodule $M$. Let $a,b$ be two distinct elements in $S$. We have
  \[ aT_{v}(1,0) = T_{v}(a,0) = (a,0) \ne (b,0) = T_{v}(b,0) = bT_{v}(1,0), \]
  so $a T_{v} \ne b T_{v}$.  The universal property of $L_{S}(\Gamma)$ then
  implies that $av \ne bv$, and (2) is established.
  
  As to (3), it follows immediately from the fact that $e^{\ast}f=\delta
  _{e,f}r(e)$ for all $e,f\in E$.
\end{proof}

As usual, for a hemiring~$S$ a \emph{set of local units}~$F$ is a set 
$F \subseteq S$ of idempotents in~$S$ such that, for every finite subset 
$\{s_{1}, \dots, s_{n}\} \subseteq S$, there exists an element $f \in F$ with
$fs_{i} = s_{i} = s_{i}f$ for all $1 \le i \le n$.  Using Proposition~2.4 and
repeating verbatim the proof of \cite[Lemma~1.6]{ap:tlpaoag05}, one obtains
the following useful fact.

\begin{prop}
  Let $\Gamma = (V,E,s,r)$ be a graph and $S$ a commutative semiring.
  Then $L_{S}(\Gamma)$ is a unital $S$-algebra if~$V$ is finite; and
  if~$V$ is infinite, the set of all finite sums of distinct elements
  of~$V$ is the set of local units of the $S$-algebra $L_{S}(\Gamma)$.
\end{prop}

Let $\Gamma = (V,E,s,r)$ be a graph.  A subset $H \subseteq V$ is called 
\emph{hereditary} if $s(e) \in H$ implies $r(e) \in H$ for all $e\in E$; 
and $H \subseteq V$ is \emph{saturated} if $v \in H$ for any regular 
vertex $\nu$ with $r(s^{-1}(v)) \subseteq H$.  Obviously, the two trivial 
subsets of $V$, $\varnothing$ and $V$, are hereditary and saturated ones. 
We note the following useful observation whose proof is completely
analogous to the ones in \cite[Lemma~3.9]{ap:tlpaoag05} and
\cite[Lemma~2.3]{ap:tlpaoag08} and which, for the reader's
convenience, we provide here.

\begin{lem}
  Let $\Gamma = (V,E,s,r)$ be a graph, $S$ a commutative semiring, and 
  $I$ an ideal of $L_{S}(\Gamma)$.  Then, $I \cap V$ is a hereditary and
  saturated subset of~$V$.
\end{lem}

\begin{proof}
  For any $e \in E$ with $s(e) \in H$ we have $r(e) \in H$, since 
  $e = s(e)e \in I$, and thus $r(e) = e^{\ast} e \in I$.  Furthermore, if a
  regular vertex $v \in V$ satisfies $r(e) \in H$ for all $e\in E$ with 
  $s(e) = v$, then $v \in H$, since $e = e r(e) \in I$ for all these edges 
  $e \in E$, and hence $v = \sum_{e\in s^{-1}(v)} e e^{\ast} \in I$.
\end{proof}

We conclude this section with the following, although simple but quite
useful, technical remark obviously following from the identity
$e^{\ast} f = \delta_{e,f} r(e)$ for all $e, f \in E$.

\begin{rem}
  For any two paths $p, q$ in $\Gamma$ we have
  \[ p^{\ast} q = \begin{cases}
    q' & \text{if } q = p q' , \\
    r(p) & \text{if } p = q , \\
    p'^{\ast} & \text{if } p = q p' , \\
    0 & \text{otherwise} .
  \end{cases} \]
\end{rem}


\section{Ideal-simpleness of Leavitt path algebras with
  coefficients in a semifield}\label{sec:ideal}

The main goal of this section is to present a description of the
ideal-simple Leavitt path algebras $L_{S}(\Gamma)$ of arbitrary graphs
$\Gamma = (V,E,s,r)$ with coefficients in a semifield $S$ that extends
the well-known description when the ground semifield~$S$ is a
field~$K$ (\cite[Theorem~3.11]{ap:tlpaoag05}, \cite[Theorem~3.1]%
{ap:tlpaoag08}, \cite[Theorem~6.18]{tomf:utaisflpa} and \cite%
[Theorem~3.11]{g:lpaadl}).  For that we have to establish some
subsequently needed important facts.

\begin{prop}
  A graph $\Gamma = (V,E,s,r)$ of an ideal-simple Leavitt path
  algebra $L_{S}(\Gamma)$ with coefficients in a commutative
  semiring~$S$ satisfies the following two conditions:
  
  (1) The only hereditary and saturated subset of~$V$ are~$\varnothing$
  and~$V$;
  
  (2) Every cycle in~$\Gamma$ has an exit.
\end{prop}

\begin{proof}
  (1) Actually the proof of the statement given in \cite[Theorem~3.11]%
  {ap:tlpaoag05} does not use the additive ring/module setting and, 
  therefore, it can be easily modified for our (nonadditive) semiring
  setting. For the reader's convenience, we briefly sketch central
  ideas of that modification here.

  Assume that~$V$ contains a nontrivial hereditary and saturated subset~$H$.
  In the same way as was shown in \cite[Theorem 3.11]{ap:tlpaoag05}, one may
  easily observe that
  \[ \Gamma' = (V',E',r_{\Gamma'},s_{\Gamma'})
  := (V\setminus H, r^{-1}(V\setminus H), r|_{V\setminus H},
  s|_{V\setminus H}) \]
  is a graph, too.  Then, as in \cite[Theorem~3.11]{ap:tlpaoag05}, let
  us consider an $S$-algebra homomorphism $\varphi: L_{S}(\Gamma) 
  \longrightarrow L_{S}(\Gamma')$ given on the generators of the free 
  $S$-algebra $A := S \langle v, e, e^{\ast} \mid v\in V, e\in E, 
  e^{\ast}\in E^{\ast} \rangle$ as follows: $\varphi(v) = \chi_{V'}(v)v$, 
  $\varphi(e) = \chi_{E'}(e)e$ and $\varphi(e^{\ast}) = \chi_{(E')^{\ast}}
  (e^{\ast})e^{\ast}$, where $\chi_{X}$ denotes the usual characteristic 
  function of a set~$X$.  To be sure that in a such manner defined map 
  $\varphi: L_{S}(\Gamma) \longrightarrow L_{S}(\Gamma')$, indeed, provides
  us with the desired hemiring homomorphism, we only need to verify that 
  all following pairs
  
  $(v v', \delta_{v, v'} v)$ for all $v, v' \in V$,
  
  $(s(e)e, e), (e, er(e))$ and $(r(e)e^{\ast}, e^{\ast}), (e^{\ast}, 
  e^{\ast}s(e))$ for all $e\in E$,
  
  $(e^* f, \delta_{e,f} r(e))$ for all $e, f\in E$,
  
  $(v, \sum_{e\in s^{-1}(v)} e e^{\ast})$ for a regular vertex $v\in V$,

  \noindent are in the kernel congruence
  \[ \text{ker}(\varphi) := \{ (x,y)\in A^{2} \mid \varphi(x)=\varphi(y) \} \]
  of $\varphi$.  But the latter can be established right away by repeating
  verbatim the corresponding obvious arguments in the proof of \cite%
  [Theorem~3.11]{ap:tlpaoag05}. Note that $|(s_{\Gamma '})^{-1}(v)| <
  \infty$ in $\Gamma'$ for any regular vertex~$v$ in $\Gamma$.  For 
  $\varnothing \neq H \varsubsetneqq V$ and Proposition~2.4, $\varphi$
  is a nonzero homomorphism and $H\subseteq \varphi^{-1}(0)$; and
  therefore, $L_{S}(\Gamma)$ contains a proper ideal and, hence, is not
  ideal-simple.

  (2) Let~$\Gamma$ contain a cycle~$p$, based at~$v$, without any exit.  
  Then, by repeating verbatim the corresponding arguments in the proof of
  \cite[Theorem~3.11]{ap:tlpaoag05}, one gets that $vL_{S}(\Gamma)v = 
  S[p,p^{\ast}]$, \textit{i.e.}, each element in $vL_{S}(\Gamma)v$ is written in 
  the form $\sum_{i=r}^s \lambda_i p^i$, where $r,s \in \mathbb{Z}$ and 
  $\lambda_i \in S$; and let $p^0 := v$ and $p^{-j} := (p^{\ast})^j$ for all 
  $j>0$.  For $L_{S}(\Gamma)$ is ideal-simple and \cite[Proposition~5.3]%
  {knz:ososacs}, $vL_{S}(\Gamma)v$ is an ideal-simple commutative semiring as 
  well.  The latter, by \cite[Theorem 11.2]{bshhurtjankepka:scs}, implies that
  $vL_{S}(\Gamma)v = S[p,p^*]$ is a semifield.  We claim that $S[p,p^*] \cong 
  S[x, x^{-1}]$, the Laurent polynomial semiring over~$S$; as this is clearly
  not a semifield, this contradiction finishes the proof.
  
  It remains to show that the natural homomorphism $S[x, x^{-1}] \to S[p, p^*]$
  given by $f \mapsto f(p)$ is, indeed, injective.  Let~$I$, $Z = S^{(I)}$, 
  $A_e$ for $e \in E$, $A_v$ for $v \in V$, and $A$ be as in the
  proof of Proposition~2.4, and consider the endomorphisms $T_v$, $T_e$,
  $T_{e^*}$ of~$A$, for $v \in V$, $e \in E$, $e^{\ast} \in E^{\ast}$.
  Without loss of generality, we may assume that $I = \mathbb{Z} \times I'$
  for some nonempty set~$I'$.
  
  Write the cycle~$p$ based at~$v$ as $p = e_1 \dots e_n$ with $e_i \in E$,
  where $v_{i-1} := s(e_i)$ and $v_i := r(e_i)$ for $1 \le i \le n$, so that
  $v_0 = v_n = v$.  By the construction in the proof of Proposition~2.4 we
  have $A_{v_{i-1}} = A_{e_i}$, since~$p$ has no exit, and $T_{e_i}$ restricts to
  an isomorphism $A_{v_i} \longrightarrow A_{v_{i-1}}$, for all~$i$.
  Consider the endomorphism $T_p := T_{e_1} \circ \dots \circ T_{e_n}$,
  which restricts to an isomorphism $T: A_v \longrightarrow A_v$,
  where $A_v = Z = S^{(I)} = S^{(\mathbb{Z} \times I')}$.  Observe that there is 
  no limitation in choosing these isomorphisms, hence we may assume that 
  $T(\delta_{(k,i)}) = \delta_{(k+1,i)}$ for all $k \in \mathbb{Z}$ and 
  $i \in I'$, where by $\delta_{(k,i)}$ we denote the standard basis vectors 
  of the free $S$-semimodule $S^{(\mathbb{Z} \times I')}$.
  
  Now suppose that $f(p) = g(p)$ for some Laurent polynomials $f, g \in
  S[x, x^{-1}]$.  Since the $T_v$, $v \in V$, $T_e$, $e \in E$, $T_{e^{\ast}}$,
  $e^{\ast} \in E^{\ast}$, satisfy the relations described in Definition~2.1, 
  it follows that $f(T) = g(T)$ holds in $\text{Hom}_S(A, A)$.  Writing 
  $f = \sum_{j=r}^s f_j x^j$ and $g = \sum_{j=r}^s g_j x^j$ for some $r, s \in 
  \mathbb{Z}$ and $f_j, g_j \in S$, from our choice of~$T$ we see that $f(T) 
  (\delta_{0,i}) = \sum_{j=r}^s f_j T^j (\delta_{0,i}) = \sum_{j=r}^s f_j 
  \delta_{j,i}$, and similarly $g(T) (\delta_{0,i}) = \sum_{j=r}^s g_j 
  \delta_{j,i}$, for any~$i \in I'$.  From this we readily deduce that 
  $f = g$ and thus the injectivity follows.
\end{proof}

Following~\cite{ap:tlpaoag05}, a monomial in $L_{S}(\Gamma)$ is a \emph{%
real path} if it contains no terms of the form $e^{\ast}\in E^{\ast}$, and
a polynomial $\alpha \in L_{S}(\Gamma)$ is in \emph{only real edges} if it
is a sum of real paths; let $L_{S}(\Gamma)_{real}$ denote the
subhemiring of all polynomials in only real edges in $L_{S}(\Gamma)$. The
following technical observation will prove to be useful.

\begin{lem}[{\textit{cf.}~\cite[Corollary~3.2]{ap:tlpaoag05}}]
  Let $\Gamma = (V,E,s,r)$ be a graph with the property that every
  cycle has an exit and~$S$ a semifield.  Then, if $\alpha \in
  L_{S}(\Gamma)_{real} \subseteq L_{S}(\Gamma)$ is a nonzero
  polynomial in only real edges, then there exist $a, b \in
  L_{S}(\Gamma)$ such that $a \alpha b \in V$.
\end{lem}

\begin{proof}
  The proof of~\cite[Corollary~3.2]{ap:tlpaoag05} does not use the
  ``additiveness'' of the setting and, therefore, repeating verbatim
  the latter, one gets the statement in our nonadditive setting as
  well. However, we provide a new proof which is much shorter than the
  Abrams and Aranda Pino's original proof.

  Namely, we write $\alpha$ in the form $\alpha = \sum_i \lambda_i
  q_i$ with $q_i$ distinct real paths and $0 \ne \lambda_i \in S$.  Out
  of the set $\{ q_i \}$ choose $p$ such that no proper prefix path
  of~$p$ is contained therein.  Let $v = r(p)$.  Then, using Remark~2.7
  we get $p^* \alpha v = \lambda v + \sum_i \lambda_i p^* q_i$, where
  the sum is over all~$q_i$ that have~$p$ as a proper prefix path and
  $r(q_i) = v$, so that $p^* q_i \in \text{CP}(v)$.

  Hence, without loss of generality, we may assume that $\alpha = \lambda v
  + \sum^n_{i=1} \lambda_i p_i$, where $p_i \in \text{CP}(v)$ of positive length
  and $0\ne\lambda \in S$. Fix some $c \in \text{CSP}(v)$.  For any $p_i \in
  \text{CP}(v)$ we may write $p_i = c^{n_i} p_i'$ with $n_i \in \mathbb N$
  maximal, so that either $p_i' = v$ or $p_i' = d_i p_i''$ with $d_i \in
  \text{CSP}(v)$, $d_i \ne c$, in which case $(c^*)^{n_i+1} p_i = c^* p_i' =
  c^* d_i p_i'' = 0$ by Remark~2.7.  With $n := \max \{n_i \mid i = 1,
  \dots, n\} + 1$, we then have that $(c^*)^n p_i c^n = p_i$ if $p_i = c^{n_i}$,
  and $(c^*)^n p_i c^n = 0$ otherwise.  Therefore, we have $\alpha' :=
  (c^*)^n \alpha c^n = \lambda v + \sum_j \lambda_j c^{n_j}$ with $n_j
  > 0$, \textit{i.e.}, $\alpha' = \lambda v + c P(c)$ for some polynomial~$P$.
  Now, we write $c$ in the form $c = e_1 \dots e_m$. By our hypothesis
  and \cite[Lemma 2.5]{ap:tlpaoag05}, there exists an exit $f\in E$ for~$c$,
  that is, there exists $j\in \{1, \dots, m\}$ such that $s(f) = s(e_j)$
  but $f\neq e_j$. Let $z:= e_1 \dots e_{j-1}f$. We get that $s(z) = v$ and
  $z^*c = 0$, so that  $\lambda^{-1}z^* \alpha' z =  z^* z + \lambda^{-1}z^*
  c P(c) z = r(z) \in V$, as desired.
\end{proof}

As was shown in \cite[Theorem~6]{col:tsiilpa}, every nonzero ideal of the
Leavitt path algebra of a row-finite graph with coefficients in a field
always contains a nonzero polynomial in only real edges. The following
observation extends this result to the Leavitt path algebra over an
arbitrary graph.

\begin{prop}
  Let $\Gamma = (V,E,s,r)$ be a graph and~$S$ a commutative semiring.
  Then any nonzero ideal~$I$ of $L_{S}(\Gamma)$ contains a nonzero
  polynomial in only real edges.
\end{prop}

\begin{proof}
  Let $I_{real} := I \cap L_{S}(\Gamma)_{real}$ for a nonzero ideal~$I$, 
  and suppose that $I_{real} = 0$.  Choose $0 \ne \alpha = \sum_{i=1}^{d}
  \lambda_{i} p_{i} q_{i}^{\ast}$ in $I$, where $d$ is minimal such that
  $p_{1}, \dots, p_{d}, q_{1}, \dots, q_{d}$ are paths in $\Gamma$ and 
  $0 \ne \lambda_{i} \in S$, $i = 1, \dots, d$.  By using \cite[Remark~3]%
  {col:tsiilpa}, as in the proof \cite[Lemma~4]{col:tsiilpa},
  one can easily get that the element~$\alpha$ can be presented in the form
  $\alpha = \mu_{1} + \dots + \mu_{m}$, where all monomials in $\mu_{j}\in I$,
  $j = 1, \dots, m$, have the same source and the same range.  Moreover,
  for $\alpha \ne 0$ and the minimality of $d$, we can assume that actually
  $\alpha = \sum_{i=1}^{d} \lambda_{i} p_{i} q_{i}^{\ast}$ with $s(p_{i}) = s(p_{j})$
  and $s(q_{i}) = s(q_{j}) = w \in V$ for all~$i$ and~$j$.  Among all such
  $\alpha = \sum_{i=1}^{d} \lambda_{i} p_{i} q_{i}^{\ast} \in I$ with minimal~$d$, 
  select one for which $(|q_{1}|, \dots, |q_{d}|)$ is the smallest in the 
  lexicographic order of $\mathbb{N}^{d}$.  Obviously, $|q_{i}| > 0$ for 
  some~$i$ (otherwise, $0 \ne \alpha \in I_{real} = 0$).  If $e \in E$, then
  \[ \alpha e = \smsum_{i=1}^{d} \lambda_{i} p_{i} q_{i}^{\ast} e
  = \smsum_{i=1}^{d'} \lambda_{i} p_{i}' (q_{i}')^{\ast}, \]
  where we either have $d' < d$, or $d' = d$ and $(|q_{1}'|, \dots, 
  |q_{d}'|)$ is smaller than $(|q_{1}|, \dots, |q_{d}|)$.  
  Whence, for the minimality of $(|q_{1}|, \dots, |q_{d}|)$, we get 
  $\alpha e = 0$ for all $e\in E$.  For $|q_{i}| > 0$ for some~$i$, we have 
  that~$w$ is not a sink, and if it is a regular vertex, we have
  \[ 0 \ne \alpha = \alpha w = \alpha \smsum_{e \in s^{-1}(w)} e e^{\ast} =
  \smsum_{e \in s^{-1}(w)} (\alpha e) e^{\ast} = 0 . \]
  Therefore, we need only to consider two possible cases when the
  vertex~$w$ emits infinitely many edges:

  \emph{Case~1.} Let $|q_{j}| > 0$ for all~$j$, and $A := \{ e \in
  s^{-1}(w) \mid q_{i}^{\ast}e \ne 0$ for some $1 \le i \le d\}$.
  Notice that $q_{i}^{\ast} e \ne 0$ if and only if the path~$q_i$
  has the form $q_i = f_1 \dots f_k$ with $k \ge 1$ and $f_1 = e$.
  Specially, in this case, we have that $q_{i}^{\ast} e e^{\ast} = 
  q_{i}^{\ast}$.  It is clear that $|A| < \infty$, and hence, $\alpha = 
  \sum_{e\in A} \alpha e e^{\ast}$.  For $\alpha e = 0$ for all $e\in E$, we 
  have $0 \ne \alpha = \sum_{e\in A} \alpha e e^{\ast} = 0$.

  \emph{Case~2.} If $|q_{j}| = 0$ for some~$j$, the element~$\alpha$ can be
  presented as
  \[ \alpha = \lambda_1 p_{1} + \dots + \lambda_m p_{m} + \lambda_{m+1} p_{m+1} 
  q_{m+1}^{\ast} + \dots + \lambda_d p_{d} q_{d}^{\ast}, \]
  where $p_1, \dots, p_m$ are distinct paths in~$\Gamma$ and
  $r(p_{i}) = w = s(q_{j})$ for all $i = 1, \dots, d$ and $j = m+1,
  \dots, d$.  Set $\beta := \lambda_1 p_{1} + \dots + \lambda_mp_{m}$.
  By Remark~2.7, we may choose a path~$p$ in $\Gamma$ such that 
  \[ p^{\ast} \beta = \lambda w + \smsum_{j=1}^k \nu_j p'_j, \]
  where $0 \ne \lambda \in S$, $\nu_j\in S$ and $p'_j \in \text{CP}(w)$ for
  all~$j$.  For~$w$ emits infinitely many edges, there is an edge $e \in
  s^{-1}(w)$ such that $q_{i}^{\ast} e = 0 = e^{\ast} p'_j$ for all $i = m+1,
  \dots, d$ and $j = 1, \dots, k$.  Then, $0 = \alpha e = \beta e \in I$
  and, hence, $p^{\ast} \beta e = \lambda e + \sum_{j=1}^k \nu_j p'_j e = 0$.
  It implies that $e^{\ast} p^{\ast} \beta e = \lambda r(e) = 0$.  Using 
  Proposition~2.4\,(2), we get that $\lambda = 0$, a contradiction.

  Hence, the ideal~$I$ contains a nonzero polynomial in only real edges.
\end{proof}

In \cite[Theorem~3.11]{ap:tlpaoag05}, the authors characterized the simple
Leavitt path algebras over countable row-finite graphs with coefficients in
a field. Then, the row-finiteness hypothesis independently was eliminated by
the authors (\cite[Theorem~3.1]{ap:tlpaoag08}) and in (\cite[Theorem~6.18]%
{tomf:utaisflpa}), and finally this characterization has been extended in
\cite[Theorem~3.11]{g:lpaadl} to arbitrary graphs. The next and main result
of this section is an extension of the latter characterization to the
Leavitt path algebras with coefficients in a semifield.

\begin{thm}
  A Leavitt path algebra $L_{S}(\Gamma)$ of a graph $\Gamma = (V,E,s,r)$ with
  coefficients in a semifield~$S$ is ideal-simple if and only if
  the graph~$\Gamma$ satisfies the following two conditions:

  (1) The only hereditary and saturated subset of~$V$ are~$\varnothing$ 
  and~$V$;

  (2) Every cycle in~$\Gamma$ has an exit.
\end{thm}

\begin{proof}
  $\Longrightarrow$. It follows from Proposition~3.1.

  $\Longleftarrow$. Let $I$ be a nonzero ideal of $L_{S}(\Gamma)$. By
  Proposition~3.3, $I$ contains a nonzero polynomial~$\alpha$ in only real
  edges.  By Lemma~3.2, there exist $a, b \in L_{S}(\Gamma)$ such that 
  $a \alpha b \in V$, \textit{i.e.}, $I \cap V \ne \varnothing$.  Now, 
  applying Lemma~2.6 and Proposition~2.5, we conclude that $I = 
  L_{S}(\Gamma)$.
\end{proof}

Taking into consideration \cite[Theorem~7.20]{tomf:lpawciacr}, the following
question seems to be reasonable, interesting and promising. \medskip

\noindent \textbf{Problem.} How far can Theorem~3.4 be extended for the
commutative ground semiring $S$? \medskip

We finish this section by demonstrating the use of Theorem~3.4 in
re-establishing the ideal-simpleness of the Leavitt path algebras of
Examples~2.3.

\begin{exas}[{\textit{cf.}~\cite[Corollary~3.13]{ap:tlpaoag05}}]
  Note that all Leavitt path algebras in these examples are algebras
  with coefficients in a semifield $S$.

  (i) By \cite[Proposition 4.7]{knt:mosssparp}, $M_{n}(S)$ is an ideal-simple
  algebra.  However, this fact can also be justified by Theorem~3.4, since it 
  is easy to check that the graph $\Gamma$ of Examples~2.3\,(i) satisfies (1)
  and (2) of Theorem~3.4.
  
  (ii) By Examples~2.3\,(ii), the Laurent polynomial algebras $S[x,x^{-1}]
  \cong L_{S}(\Gamma)$ where the graph $\Gamma$ contains a cycle without an 
  exit, and therefore, by Theorem~3.4, $S[x,x^{-1}]$ is not ideal-simple.
  
  (iii) By Examples~2.3\,(iii), the Leavitt algebras $L_{S}(1,n)$ for $n\geq 2$
  are isomorphic to the Leavitt path algebras $L_{S}(\Gamma)$ such that for
  the graphs~$\Gamma$ conditions~(1) and~(2) of Theorem~3.4 are obviously
  satisfied, and therefore, the algebras $L_{S}(1,n)$ are ideal-simple.
  (Note that we consider here an $S$-algebra analog of a Leavitt algebra over
  a field, see \cite[Theorem~2]{leav:tmtohi}).
\end{exas}


\section{Congruence-simpleness of Leavitt path algebras
  with coefficients in a commutative semiring}%
\label{sec:congr}

Providing necessary and sufficient conditions for a Leavitt path algebra
over a row-finite graph with coefficients in a commutative semiring to be
congruence-simple is the main goal of this section. We start with necessary 
conditions for such algebras to be congruence-simple, namely:

\begin{prop}
  For a congruence-simple Leavitt path algebra $L_{S}(\Gamma)$ of a
  graph $\Gamma = (V,E,s,r)$ with coefficients in a commutative
  semiring $S$ the following statements are true:

  (1) $S$ is either a field, or the Boolean semifield~$\mathbf{B}$;
  
  (2) The only hereditary and saturated subset of~$V$ are~$\varnothing$
  and~$V$;
  
  (3) Every cycle in~$\Gamma$ has an exit.
\end{prop}

\begin{proof}
  (1) First, let us show that there are only the two trivial congruences 
  on~$S$.  Indeed, if $\sim$ is a proper congruence on~$S$, the natural 
  surjection $\pi: S \longrightarrow \overline{S} := S / {\sim}$, defined
  by $\pi(\lambda) = \overline{\lambda}$, is neither zero nor an injective
  homomorphism.  As one can easily verify, the homomorphism~$\pi$ induces
  a nonzero surjective hemiring homomorphism $\varphi: L_{S}(\Gamma)
  \longrightarrow L_{\overline{S}}(\Gamma)$ such that $\varphi (\lambda p
  q^{\ast}) = \overline{\lambda} p q^{\ast}$, where $\lambda \in S$ and
  $p, q$ are paths in $\Gamma$ with $r(p) = r(q)$.  For $\pi$ is not 
  injective, there exist two distinct elements $a, b \in S$ such that 
  $\overline{a} = \overline{b}$ and, by Proposition~2.4\,(2), $a v \ne b v$
  in $L_{S}(\Gamma)$ for any $v\in V$. However,
  \[ \varphi (a v) = \overline{a} v = \overline{b} v = \varphi (b v) , \]
  and hence, $\varphi$ is not injective, and therefore, $L_{S}(\Gamma)$ is
  not congruence-simple. Thus, $S$ is congruence-simple, and it follows by
  \cite[Theorem~10.1]{bshhurtjankepka:scs} (see also \cite[Theorem~3.2]%
  {mf:ccs}) that~$S$ is either a field, or the semifield~$\mathbf{B}$.
  
  (2) A proof of this statement is established by the proof of
  Proposition~3.1\,(1): Indeed, in the notation of the latter, one readily
  concludes that the map $\varphi: L_{S}(\Gamma) \longrightarrow
  L_{S}(\Gamma')$ is a nonzero homomorphism and $H\subseteq \varphi
  ^{-1}(0)$, and hence, $L_{S}(\Gamma)$ is not congruence-simple.
  
  (3) This statement can be proven analogously to the proof of
  Proposition~3.1\,(2); and in the notations of the latter, one readily
  concludes that $v L_S(\Gamma) v = S[p,p^*] \cong S[x, x^{-1}]$, the Laurent
  polynomial semiring over~$S$.  By \cite[Proposition~5.3\,(2)]{knz:ososacs},
  $vL_{S}(\Gamma)v$ is a congruence-simple semiring; that means, $S[x, x^{-1}]$
  is congruence-simple, too.  This would imply, by \cite[Theorem~10.1]%
  {bshhurtjankepka:scs} (see also \cite[Theorem~3.2]{mf:ccs}), that
  $S[x, x^{-1}]$ is either a field or the Boolean semifield $\mathbf{B}$,
  what is obviously is not a case.
\end{proof}

Combining Theorem 3.4 and Proposition 4.1, one immediately obtains that the
congruence-simpleness of a Leavitt path algebra over an arbitrary graph with
coefficients in a commutative semiring implies its ideal-simpleness, what,
in turn, actually resolves \cite[Problem 2]{knt:mosssparp} in the class of
Leavitt path algebras, namely:

\begin{cor}
  A congruence-simple Leavitt path algebra $L_{S}(\Gamma)$ over an
  arbitrary graph~$\Gamma$ with coefficients in a commutative
  semiring~$S$ is ideal-simple as well.
\end{cor}

Next, modifying the ideas and techniques used in the proof of~\cite%
[Theorem~6]{col:tsiilpa}, we obtain a semiring version of this result
for the Leavitt path algebras over the Boolean semifield~$\mathbf{B}$.

\begin{prop}
  Let $\Gamma = (V,E,s,r)$ be a row-finite graph, $\rho$ a congruence
  on $L_{\mathbf{B}} (\Gamma)$, and $\rho_{real} := \rho \cap 
  (L_{\mathbf{B}} (\Gamma)_{real})^{2}$.  Then~$\rho$ is generated by
  $\rho_{real}$.
\end{prop}

\begin{proof}
  Let $\tau$ be the congruence on $L_{\mathbf{B}}(\Gamma)$ generated by
  $\rho_{real}$, then the inclusion $\tau \subseteq \rho$ is obvious.
  Suppose that $\tau \ne \rho$, i.e., there exists $(x, y) \in \rho$
  with $(x, y) \notin \tau$.  By Proposition~2.5 we may choose a
  finite subset $F \subseteq V$ such that $x = x \sum_{v \in F} v$ and
  $y = y \sum_{v \in F} v$, and therefore
  \[ (x, y) = ( x \smsum_{v \in F} v, y \smsum_{v \in F} v ) =
  \smsum_{v \in F} ( x v, y v) . \]
  Since $(x, y) \notin \tau$, there exists $v \in F$ such that $(x v, y v)
  \notin \tau$, and we have $(x v, y v) \in \rho$.  Therefore, we may assume
  that $x = \sum_{i=1}^{k} p_{i} q_{i}^{\ast}$ and $y = \sum_{j=1}^{l}
  \gamma_{j} \delta_{j}^{\ast}$ with $p_{i}, q_{i}, \gamma_{j}, \delta_{j}$
  paths in~$\Gamma$ and $r(q_{i}^{\ast}) = r(\delta_{j}^{\ast}) = v$ for
  all~$i,j$.  Among all such pairs $(\sum_{i=1}^{k} p_{i} q_{i}^{\ast} ,\,
  \sum_{i=1}^{l} \gamma_{i} \delta_{i}^{\ast}) \in \rho \setminus \tau$ with
  minimal $d:=k+l$, select one for which $(|q_{1}|, \dots, |q_{k}|,
  |\delta_{1}|, \dots, |\delta_{l}|)$ is the smallest in the lexicographic 
  order of~$\mathbb{N}^{d}$. As $(x,y) \notin \tau$, one has $|q_{i}| > 0$ 
  for some~$i$, or $|\delta_{j}|>0$ for some~$j$. For all $e \in E$,
  \[ (xe,ye) = (\smsum_{i=1}^{k} p_{i}q_{i}^{\ast}e,
  \smsum_{i=1}^{l} \gamma_{i}\delta_{i}^{\ast}e)
  = (\smsum_{i=1}^{k'} p_{i}' (q_{i}')^{\ast},
  \smsum_{i=1}^{l'} \gamma_{i}' (\delta_{i}')^{\ast}) , \]
  and either $d' := k' + l' < d$, or $d' = d$ and $(|q_{1}'|, \dots, |q_{k}'|,
  |\delta_{1}'|, \dots, |\delta_{l}'|)$ is smaller than $(|q_{1}|, \dots, 
  |q_{k}|, |\delta_{1}|, \dots, |\delta_{l}|)$, whence $(xe,ye) \in \tau$,
  by minimality.  As some $|q_{i}| > 0$ or some $|\delta_{j}| > 0$, 
  it follows that~$v$ is not a sink, and hence, $(x,y) = (x v,y v) = 
  (x \sum_{e\in s^{-1}(v)}ee^{\ast}, y \sum_{e\in s^{-1}(v)}ee^{\ast}) = 
  \sum_{e\in s^{-1}(v)} ((xe)e^{\ast}, (ye)e^{\ast}) \in \tau$, contradicting 
  that $(x,y) \notin \tau$.  This shows that $\rho = \tau$.
\end{proof}

The following result, being an $\mathbf{B}$-algebra analog of
\cite[Theorem~3.11]{ap:tlpaoag05}, characterizes the congruence-simple
Leavitt path algebras over the Boolean semifield $\mathbf{B}$.

\begin{thm}
  A Leavitt path algebra $L_{\mathbf{B}}(\Gamma)$ of a row-finite
  graph $\Gamma = (V,E,s,r)$ is congruence-simple if and only if the
  graph~$\Gamma$ satisfies the following two conditions:
  
  (1) The only hereditary and saturated subset of~$V$ are~$\varnothing$
  and~$V$;
  
  (2) Every cycle in~$\Gamma$ has an exit.
\end{thm}

\begin{proof}
  $\Longrightarrow$. It follows from Proposition~4.1.

  $\Longleftarrow$. Let $\rho \neq \Delta_{L_{\mathbf{B}}(\Gamma)}$ be a
  congruence on $L_{\mathbf{B}}(\Gamma)$. Then, by Proposition~4.3, $\rho$
  is generated by $\rho_{real} := \rho \cap (L_{\mathbf{B}}(\Gamma)_{real})^{2}$
  and $\rho_{real} \neq \Delta_{L_{\mathbf{B}}(\Gamma)_{real}}$.  Hence,
  there exist two elements $a, b \in L_{\mathbf{B}}(\Gamma)_{real}$ such that
  $a\neq b$ and $(a,b) \in \rho$.  We claim that there exists a nonzero
  polynomial $x \in L_{\mathbf{B}}(\Gamma)$ in only real edges such that
  $(0,x) \in \rho$.

  It is clear that $L_{\mathbf{B}}(\Gamma)$ is an additively idempotent
  hemiring, \textit{i.e.}, $L_{\mathbf{B}}(\Gamma)$ is a partially ordered
  hemiring with its unique partial order defined as follows: $s \leq s'
  \Longleftrightarrow s + s' = s'$.  Whence, $(a,a+b) = (a+a,a+b) \in \rho$,
  $(b,a+b) = (b+b,a+b) \in \rho$, and since $a \ne b$, either $a < a+b$ or
  $b < a+b$.  Thus, keeping in mind that $(a+x,b+x) \in \rho$ for all 
  $x \in L_{\mathbf{B}}(\Gamma)$ and without loss of generality, one may 
  assume that $a < a + b$ and $a$, $a+b$ are written in the form
  \[ a = p_{1} + \dots + p_{n}, \quad
  a+b = p_{1} + \dots + p_{n} + p , \]
  where $p_{1}, \dots,  p_{n}, p$ are distinct paths in $\Gamma$.
  Moreover, we may choose~$a$ having the minimal number~$n$ of such
  $\{p_{1} , \dots,  p_{n}\}$.

  Let $v := s(p)$, $w := r(p) \in V$.  Then $(v a w, v (a+b) w) \in
  \rho$, where $v a w = v p_1 w + \dots + v p_n w$ and $v b w = v p w
  = p$, hence by minimality we may assume that $s(p_i) = v$ and
  $r(p_i) = w$ for all $i = 1, \dots, n$.

  Suppose that $v \ne w$.  Write $p = q p'$, where~$q$ is a path
  from~$v$ to~$w$ of shortest length and~$p'$ is a closed path based
  at~$w$.  Taking into account Remark~2.7, for every~$p_{j}$ such that
  $q^{\ast} p_{j} \ne 0$ we have $p_{j} = q p_{j}'$ for some closed
  path~$p_{j}'$ based at~$w$.  Then we have $(q^{\ast} a, q^{\ast} (a+b))
  = (q^{\ast} p_{1} +\dots+ q^{\ast}p_{n}, q^{\ast}p_{1} +\dots+ q^{\ast}p_{n}
  + q^{\ast}p) = (\sum_{j} p_{j}', \sum_{j} p_{j}' + p') \in \rho$.  
    Therefore, without loss of generality, we may assume that $v = w$,
  \textit{i.e.}, $p, p_{1} , \dots, p_{n}$ are distinct closed paths
  based at~$v$, and consider the following two possible cases.

  \emph{Case~1.} There exists exactly one closed simple path based at~$v$,
  say $c := e_{1} \dots e_{m}$.  It follows that~$c$ is actually a cycle, and
  by condition~(2), $c$ has an exit~$f$, \textit{i.e.}, there exists
  $j \in \{1, \dots, m\}$ such that $e_{j}\neq f$ and $s(f)=s(e_{j})$.  Then,
  there are some distinct positive integers $k$, $k_{i}$, $i=1, \dots, n$,
  such that $p =c^{k}$ and $p_{i} = c^{k_{i}}$, $i=1, \dots, n$, and let
  \begin{gather*}
    x := (c^{\ast})^{k}a = (c^{\ast})^{h_{1}} +\dots+ (c^{\ast})^{h_{r}}
    + c^{h_{r+1}} +\dots+ c^{h_{n}} \\
    y := (c^{\ast})^{k}(a+b) = (c^{\ast})^{h_{1}} +\dots+ (c^{\ast})^{h_{r}}
    + c^{h_{r+1}} +\dots+ c^{h_{n}} + v.
  \end{gather*}%
  Obviously, $(x,y)\in \rho$, and therefore, $(0, r(f)) = (z^{\ast} x z,
  z^{\ast} y z ) \in \rho$ for $z := e_{1} \dots e_{j-1} f$.

  \emph{Case~2.} There exist at least two distinct closed simple paths
  based at~$v$, say~$c$ and~$d$, and we have $c^{\ast} d = 0 = d^{\ast} c$
  by Remark~2.7.  Note that $(p^{\ast} a, p^{\ast} (a+b) ) \in \rho$
  and let
  \begin{gather*}
    x := p^{\ast} a = q_{1}^{\ast} +\dots+ q_{s}^{\ast} + q_{s+1} +\dots+ q_{n} \\
    y := p^{\ast}(a+b) = q_{1}^{\ast} +\dots+ q_{s}^{\ast} + q_{s+1} +\dots+ q_{n}
    + v ,
  \end{gather*}%
  where $q_{1}, \dots, q_{n}$ are closed paths in~$\Gamma$ based at~$v$.
  Then for some $k\in \mathbb{N}$, where $|c^{k}| > \max \{|q_{1}| , \dots, 
  |q_{n}|\}$, we get $x' := (c^{\ast})^{k} x c^{k} = (c^{\ast})^{k}
  q_{1}^{\ast}c^{k} +\dots+ (c^{\ast})^{k} q_{s}^{\ast}c^{k} +
  (c^{\ast})^{k}q_{s+1}c^{k} +\dots+ (c^{\ast})^{k}q_{n}c^{k}$ and
  $y' := (c^{\ast})^{k} y c^{k} = (c^{\ast})^{k}q_{1}^{\ast}c^{k}
  +\dots+ (c^{\ast})^{k} q_{s}^{\ast}c^{k} + (c^{\ast})^{k}q_{s+1}c^{k}
  +\dots+ (c^{\ast})^{k}q_{n}c^{k} + v$, and $(x',y') \in \rho$.
  If $(c^{\ast})^{k} q_{i}^{\ast}c^{k} = 0 = (c^{\ast})^{k} q_{j}c^{k}$
  for all $i=1, \dots, s$ and $j=s+1, \dots, n$, then $(0,v) = (x',y')
  \in \rho$.  Note that if $(c^{\ast})^{k} q_{j}c^{k} \neq 0$, then
  $(c^{\ast})^{k} q_{j}\neq 0$, and as $|c^{k}| > |q_{j}|$, $c^{k}
  = q_{j} q_{j}'$ for some closed path $q_{j}'$
  based at~$v$.  Whence, $q_{j} = c^{r}$ for some positive integer
  $r\leq k$.  Similarly, in the case $(c^{\ast})^{k} q_{i}^{\ast}c^{k}
  \neq 0$, we get that $q_{i}^{\ast} = (c^{\ast})^{s}$ for some positive
  integer $s\leq k$.  Since $c^{\ast} d = 0 = d^{\ast} c$, for every $i,j$, one
  gets $d^{\ast} (c^{\ast})^{k} q_{i}^{\ast} c^{k} d = 0 = d^{\ast} (c^{\ast})^{k}
  q_{j} c^{k} d$, and hence, $(0,v) = (d^{\ast} x' d,
  d^{\ast} y' d) \in \rho$.
  
  Finally, let us consider the ideal of $L_{\mathbf{B}}(\Gamma)$ defined as
  follows:
  \[ I := \{ x\in L_{\mathbf{B}}(\Gamma) \mid (0,x) \in \rho \} . \]
  From the observations above, $I$ contains a nonzero polynomial in
  only real edges. By our assumption and Theorem 3.4, $L_{\mathbf{B}}
  (\Gamma)$ is an ideal-simple hemiring, and hence, $I = L_{\mathbf{B}}
  (\Gamma)$. It immediately follows that $\rho = L_{\mathbf{B}}(\Gamma)^{2}$,
  which ends the proof.
\end{proof}

Combining Proposition~4.1, Theorem~4.4 and \cite[Theorem~3.11]%
{ap:tlpaoag05}, we obtain a complete characterization of the
congruence-simple Leavitt path algebras $L_{S}(\Gamma)$ of row-finite
graphs~$\Gamma$ over commutative semirings.

\begin{thm}
  A Leavitt path algebra $L_{S}(\Gamma)$ of a row-finite graph $\Gamma
  = (V,E,s,r)$ with coefficients in a commutative semiring~$S$ is
  congruence-simple if and only if the following three conditions are
  satisfied:
  
  (1) $S$ is either a field, or the Boolean semifield~$\mathbf{B}$;

  (2) The only hereditary and saturated subset of~$V$ are~$\varnothing$
  and~$V$;
  
  (3) Every cycle in~$\Gamma $ has an exit.
\end{thm}

In light of \cite[Theorem~3.1]{ap:tlpaoag08}, \cite[Theorem~6.18]%
{tomf:utaisflpa} and \cite[Theorem~3.11]{g:lpaadl}, and to stimulate
an interest of some potential readers in research in this, in our
view, quite interesting and promising direction, we post the
following\medskip

\noindent \textbf{Conjecture}. Theorem~4.5 is true for the Leavitt path
algebras $L_{S}(\Gamma)$ over an arbitrary graph $\Gamma$. \medskip

As was done in the previous section, we end this section and the paper
by re-establishing the congruence-simpleness of the Leavitt path
algebras of Examples~2.3.

\begin{exas}[{\textit{cf.}~\cite[Corollary~3.13]{ap:tlpaoag05}}]
  We can re-establish the congruence-simpleness of the algebras given
  in Example~3.2 above.

  (i) By \cite[Corollary~4.8]{knt:mosssparp}, $M_{n}(\mathbf{B})$ is
  congruence-simple.  However, this fact can be also justified by
  Theorem~4.5, since it is easy to check that the graph~$\Gamma$ of
  Examples~2.3\,(i) satisfies~(1) and~(2) of Theorem~4.5.

  (ii) By Examples~2.3\,(ii), the Laurent polynomial algebra
  $\mathbf{B}[x,x^{-1}]\cong L_{S}(\Gamma)$ where the graph $\Gamma$
  contains a cycle without an exit, and therefore, by Theorem~4.5,
  $\mathbf{B}[x,x^{-1}]$ is not congruence-simple.

  (iii) By Examples~2.3\,(iii), the Leavitt algebras $L_{\mathbf{B}}(1,n)$
  for $n \ge 2$ are isomorphic to the Leavitt path algebras
  $L_{\mathbf{B}}(\Gamma)$ such that for the graphs~$\Gamma$
  conditions~(1) and~(2) of Theorem~4.5 are obviously satisfied, and
  therefore, the algebras $L_{\mathbf{B}}(\Gamma)$ are congruence-simple.
\end{exas}


\end{document}